\documentclass{article}
\usepackage{amsmath, amsthm, amssymb}
\usepackage{graphicx}
\usepackage{path}
\usepackage[all]{xy}

\newcommand{\U}{\mathfrak{U}}
\newcommand{\K}{\mathbb{K}}
\newcommand{\lt}{\mathrm{lt}}
\newcommand{\lm}{\mathrm{lm}}
\newcommand{\supp}{\mathrm{supp}}
\newcommand{\Z}{\mathbb{Z}}
\newcommand{\C}{\mathbb{C}}

\newcommand{\N}{\mathbb{N}}
\newcommand{\F}{\mathbb{F}}

\newcommand{\chr}{\operatorname{char}}

\newcommand{\Rad}{\operatorname{Rad}}
\newcommand{\norm}{\operatorname{N}}
\newcommand{\Deg}{\operatorname{Deg}}
\newtheorem{definition}{Definition}
\newtheorem{proposition}{Proposition}
\newtheorem{theorem}{Theorem}
\newtheorem{corollary}{Corollary}
\newtheorem{lemma}{Lemma}

\author{Ivan Yudin
\\ CMUC, Departamento de Matematica\\ Universidade de Coimbra\\
3001-454 Coimbra\\ Portugal\\
\texttt{yudin@mat.uc.pt}
\thanks{The work is supported by the FCT Grant SFRH/BPD/31788/2006 
and
by CMUC and FCT (Portugal), through
European program
COMPETE/FEDER. } }

\title{Gr\"obner basis and the Anick resolution for $\U_{\K}(sl^+_3)$.}

\begin{document}
\maketitle
\section{Introduction}

Despite extension groups between modules over an algebra are very easy to define
and taught nowadays in every standard course in homological algebra, it is still
to be very difficult to compute them explicitly for a given pair of modules. One
of such problems is a computation of extension groups between Weyl modules over
the Schur algebra $S(n,r)$. It was shown in the joint work~\cite{apsme} of the
author with Ana Paula Santana that this problem is closely related to the
construction of a minimal projective resolution of the trivial module $\K$ over
Kostant form $\U_\K(sl^+_n)$ of the universal enveloping algebra of the  Lie algebra $sl^+_n$. 

In this paper we compute the first three steps of a minimal projective
resolution of $\K$ for  $n=3$. For this we use the Anick
resolution constructed in~\cite{anick}. Our result depends on the knowledge of
a Gr\"obner basis for $\U_K(sl^+_n)$.  

In the Section~\ref{basis} we recall the definition of Gr\"obner basis and in
the Section~\ref{Anickres} the construction of the Anick's resolution. 
Then we proceed with the definition of $\U_\K(sl^+_n)$ in Section~\ref{kostant}. 
The Sections~\ref{bigbasis},~\ref{smallbasis},~\ref{firststeps} contain new
results.  
In particular, we describe the first three steps of the minimal projective
resolution for trivial module over $\U_\K\left( sl_3^+ \right)$.

\section{Gr\"obner basis}
\label{basis}
Let $X$ be a set. We denote by $X^*$ the set of all words with letters
in $X$. Then $X^*$ is a \emph{free monoid} generated by $X$ with the
multiplication given by concatenation of words and the unity $e$ given
by the empty word. 
There is a partial order $\prec$ on $X^*$ given by the incusion of words. Note
that $\prec$ is the coarsest partial order on $X^*$ such that $X^*$ is an
ordered monoid with $e$ the least element of $X^*$. A \emph{monoidal order} on
$X^*$ is a total order that refines $\prec$. 

Let $\K$ be a field. We denote by $\K\left\langle X^* \right\rangle$ a vector
space spanned by $X^*$. A vector space $\K\left\langle X^* \right\rangle$ is a
\emph{free associative algebra} generated by $X$. 
We will call the elements of $X^*$ \emph{monomials}, and the elements of
$\K\left\langle X^* \right\rangle$ \emph{polynomials}.
Define the support of $p\in \K\left\langle X^* \right\rangle$ to be the set of
element in $X^*$ with non-zero coefficients in $p$. If $\le$ is a monoidal order
on $X^*$ then we define the \emph{ leading monomial} $\lm(p)$ of $p\in
\K\left\langle X^* \right\rangle$ to be the maximal element of support of
$p$ with respect $\le$. Define the  \emph{leading term} $\lm(p)$ of $p$ to be the leading
monomial of $p$ with coefficient it enters in $p$. A monoidal order $\le$ on
$X^*$ can be extended to a partial order $\le$ on $\K\left\langle X^*
\right\rangle$ by the rule
\begin{align*}
	p\le q \Longleftrightarrow & \lm(p)<\lm(q)\\
	& \lt(p)=\lt(q) \mbox{ and } p-\lt(p) \le q-\lt(q).
\end{align*}
Note that in the case $\lm(p)=\lm(q)$ but $\lt(p)\not=\lt(q)$ the polynomials 
$p$ and $q$ are incompatible.

The pair $(m,f)$, where $m$ is a monomial and $f$ an element of $\K\left\langle
X^* \right\rangle$, is called a \emph{rewriting rule} if $m> f$.
Note that every element $p\in \K\left\langle X^* \right\rangle$ gives a
rewriting rule $r(p)=(\lm(p),f)$ where $f=(p-\lt(p))/\lambda$ and $\lambda$ is the
leading coefficient of $p$. 
We will say that 
$h$ is a result of application of $(m,f)$ to $g$ if there is $m'\in\supp(g)$
such that  $m'=umv$ for some $u$, $v\in X^*$, and
$h=g-\lambda m'+\lambda ufv$, where $\lambda$ is the coefficient of $m$ in
$g$. We will write in this situation $g \to_{r} h$. If $r=r(p)$ for some
$p\in \K\left\langle X^* \right\rangle$ then we write $g\to_f h$ instead of
$g \to_{r(p)} h$. Let $S$ be a collection of rewriting rules or polynomials. Then 
$g\to_S h$ denotes that there is $r\in S$ such that $g\to_r h$. Formally, $\to_S$ is a
set relation on $\K\left\langle X^* \right\rangle$. We denote by $\to_S^*$ the
reflexive and transitive closure of $\to_S$.  
An element $g$ of $\K\left\langle X^* \right\rangle$ is called
\emph{non-reducible} with respect to the set of rewriting rules or polynomials
$S$ if $g$ is a minimal element
of $\K\left\langle X^* \right\rangle$ with respect to $\to_S^*$.

\begin{definition}
Let $A$ be an algebra over a field $\K$ and $X=\left\{\, a_i \,\middle|\, i\in I
\right\}$ a set of generators of $A$. Denote by $\pi$ the canonical projection
from $\K\left\langle X^* \right\rangle$ to $A$. We say that a subset $S$ of
$\ker\left( \pi \right)$ is a \emph{Gr\"obner basis} of $\ker(\pi)$ if
$\pi$ restricted on the vector space of non-reducible elements with respect
$\left\{\, r(p) \,\middle|\, p\in S \right\}$ is an isomorphism of $\K$-vector
spaces. A Gr\"obner basis $S$ is called \emph{reduced} if elements $p\in S$
are non-reducible with respect to $S\setminus \left\{ p
\right\}$.
\end{definition}
Suppose that $\le$ is an artinian monoidal order on $X^*$, that is every
descending chain in  $X^*$ stabilizes. Let $f\in \K\left\langle X^* \right\rangle$. If
$f$ is reducible with respect to a Gr\"obner basis then there is $f_1$ such that
$f\to_S f_1$. By definition of Gr\"obner basis $f_1<f$ with respect to
the induced ordering on $\K\left\langle X^* \right\rangle$. If $f_1$ is
reducible we can find $f_2$ such that $f_1\to_S f_1$, $f_1>f_2$ and so on. Thus
we get a descending sequence $f>f_1>f_2>\dots$. As we assumed that the ordering
$\le$
is artinian this sequence have to break. Thus there is $f'$ that is
non-reducible with respect to $S$ and $f\to_S f'$. We call $f'$ the normal
form of $f$ with respect to $S$ and denote it by $NF(f,S)$. Note that the use of
the article ``the'' is justified by the fact that $f'$ is unique. In fact
suppose there are $f'$ and $f''$ such that $f\to_S f'$ and $f\to_S f''$. Then 
$f'-f''=(f'-f)+(f-f'')\in\ker(\pi)$ is an element of the kernel of the natural
projection $\pi\colon \K\left\langle X^* \right\rangle\to A$. Moreover, all
monomials in $f'-f''$ are non-reducible with respect to $S$. Since the images of
non-reducible monomials with respect to $S$ give a basis of $A$ under the map
$\pi$ it immediately follows that $f'-f''=0$.

The notion of Gr\"obner basis is closely connected with the notion of
critical pairs. We say that two monomials $m_1$, $m_2\in X^*$ \emph{overlaps} if there are
$u$, $v$, $w\in X^*$ such that $m_1= uv$ and $m_2=vw$. Note that two given
monomials can have different overlappings. To make things more convenient we define
an \emph{overlapping} as a triple $(m,m_1,m_2)$, such that there are $u$,
$v\in X^*$ such that $m=m_1v$ and $m=um_2$. 

\begin{definition}
	A \emph{critical pair} is a triple $(w,r_1,r_2)$, where $w$ is a
word	
	and $r_1=(m_1,f_1)$, $r_2=(m_2,f_2)$ are rewriting rules such that 
	there are $u$, $v\in X^*$ with the property
	$$
	w=um_1=m_2v \mbox{ or } w=um_1v = m_2.
	$$
	A word
	$w$ is called the \emph{tip} of the critical pair $(w,r_1,r_2)$.
\end{definition}

Let $(w,r_1,r_2)$ be a critical pair with $r_1$, $r_2\in S$ and $u$, $v\in X^*$
such that $w=um_1=m_2v$ (or $w=um_1v=m_2$). It is called reducible if $uf_1-f_2v
\to_S^* 0$ (respectively $uf_1v-f_2\to_S^* 0$). 
The set of rewriting rules $S$ is called complete if all critical pairs
$(w,r_1,r_2)$ with $r_1$, $r_2\in S$ are reducible.

\begin{theorem}
	\label{groebner}Suppose $\le$ is artinian monoidal ordering on
	$X^*$. A subset $S$ of $\K\left\langle X^* \right\rangle$ is a Gr\"obner
	basis of a two-sided ideal $I\subset \K\left\langle X^*
	\right\rangle$ if and only if the set of rewriting rules $\left\{\,
	r(p)
	\,\middle|\, p\in S \right\}$ is complete. 
\end{theorem}

We shall need the following proposition
\begin{proposition}
 \label{subsystem} Suppose $R$ is a complete rewriting system in
 variables $X$ and $Y$ is a subset of $X$. We denote by $R(Y)$ the
 subset of $R$ that consist from all the rules $(m,p)$ such that
 $m\in Y^*$. If for all $(m,p)\in R(Y)$ we have $p\in \K\left\langle
 Y^* \right\rangle$ then $R(Y)$ is a complete rewriting system. 
\end{proposition}
\begin{proof}
 Suppose $f\in \K\left\langle Y^*\right\rangle$ and $f\to_{R} g$
 then $f\to_{(m,p)} g$ for some $(m,p)\in R$. Since $m\preceq m'$ for
 some $m'\in \supp(f)$ and $m'\in Y^*$ we get that $(m,p)\in
 R(Y)$. By assumption of the proposition we get $p\in \K\left\langle
 Y^*\right\rangle$. Therefore $g\in \K\left\langle Y^*\right\rangle$
 and $f\to_{R(Y)} g$. Now by repetition we get that
 $f\in\K\left\langle Y^*\right\rangle$ and $f\to_{R}^* g$ implies that
 $f\to_{R(Y)}^* g$. 

 Suppose that $(w,r_1,r_2)$ is an overlap of two rules from
 $R(Y)$ and $u$, $v\in Y^*$ are such that $w=m_1v=um_2$
 ($w=um_1v=m_2$). Then $p_1v-up_2\in \K\left\langle Y^*\right\rangle$
 ($up_1v-p_2\in \K\left\langle Y^*\right\rangle$) and $p_1v-up_2\to_R
 0$ ($up_1v-p_2\to_R 0$), since $R$ is complete. But then
 $p_1v-up_2\to_{R(Y)}
 0$ ($up_1v-p_2\to_{R(Y)} 0$), which shows that $R(Y)$ is complete.
\end{proof}

\section{Konstant form of universal enveloping algebra}
\label{kostant}
Denote by $sl_3^+$ the Lie algebra of upper triangular  nilpotent $3\times 3$
matrices.
Let $\U_3^+(\C)$ be its universal enveloping algebra over $\C$. 
 We shall consider $sl^{+}_3$ with the standard basis 
 \begin{align*} 
  e_{\alpha}& = 
   \left(
    \begin{smallmatrix}
    	 0 & 1 & 0 \\
 	 0 & 0 & 0 \\
 	 0 & 0 & 0 
  \end{smallmatrix}
   \right) & 
   e_\beta& = 
   \left( 
   \begin{smallmatrix}
	   0 & 0 & 0\\
	   0 & 1 & 0 \\
	   0 & 0 & 0 
   \end{smallmatrix}
   \right)& 
   e_{\alpha+ \beta}& = 
   \left( 
   \begin{smallmatrix}
	   0 & 0 & 1 \\
	   0 & 0 & 0 \\
	   0 & 0 & 0 
   \end{smallmatrix}
   \right). 
   \end{align*}
   They also generate $\U_3^+\left( \C \right)$ as an associative algebra. 
It follows from the Poincare-Birkhoff-Witt Theorem, that 
the set 
$$
\mathbb{B} =\left\{e_{\alpha}^{k_\alpha} e_{\alpha+\beta}^{k_{\alpha+\beta}}
e_{\beta}^{k_\beta}\middle| k_{\alpha}, k_{\alpha+\beta}, k_{\beta}\in \N
\right\}
$$
is a $\C$-basis of $\U_3^+(\C)$.
For $\omega\in \left\{ \alpha, \alpha+\beta, \beta \right\}$,  denote by $e_{\omega}^{(k)}$
the element $\frac{1}{k!}e_{\omega}^{k}$ of the algebra $\U_n(\C)$.
We define $\U_3^+(\Z)$ to be the $\Z$-sublattice of $\U_3^+(\C)$ generated by
the set 
$$
\overline{\mathbb{B}} =\left\{ e_{\alpha}^{\left(k_\alpha\right)}
e_{\alpha+\beta}^{\left(k_{\alpha+\beta}\right)}
e_{\beta}^{\left(k_\beta\right)}\middle| k_{\alpha}, k_{\alpha+\beta}, k_{\beta}\in \N
\right\}.
$$
\begin{proposition}
 The set $\U_3^+(\Z)$ is a subring of $\U_3^+(\C)$. In other words,
 $\U_3^+(\Z)$  is a $\Z$-algebra. It is called the \emph{Kostant form}
 of the universal enveloping algebra $\U_3^+(\C)$ over $\Z$. 
\end{proposition}
\begin{proof}
 For a proof see~\cite[Lemma~2 after Proposition~3]{kostant}
 and~\cite[Remark~3]{kostant} thereafter. 
\end{proof} 

\begin{definition}
 For any field $\K$, the algebra $\U_3^+(\K) := \K\otimes_{\Z} \U_3^+(\Z)$ is called
 \emph{Kostant form} of the algebra $\U_3^+(\C)$ over 
 $\K$.
\end{definition}
Let $S$ be a free monoid generated by $\alpha$ and $\beta$. 
We define $S$-grading on $sl_3^+\left( \C \right)$ by 
\begin{align*}
	\deg\left( e_{\alpha} \right) &: = \alpha & \deg\left( e_{\alpha+\beta}
	\right)& := \alpha + \beta & \deg\left( e_{\beta} \right) & := \beta. 
\end{align*}

 This grading extends to the grading of $\U_3^+(\C)$
 by
 $$
 \deg\left(  e_{\alpha}^{k_\alpha} e_{\alpha+\beta}^{k_{\alpha+\beta}}
 e_{\beta}^{k_\beta}\right): = \left(k_\alpha + k_{\alpha+\beta}\right)\alpha +
 \left(k_{\alpha+\beta} + k_{\beta}\right)\beta. 
 $$
 It induces a grading on the algebras
  $\U_3^+\left( \Z \right)$ and $\U_3^+\left( \K \right)$, for an arbitrary
 field $\K$, such that 
$$
\deg\left( e_{\alpha}^{\left(k_\alpha\right)}
e_{\alpha+\beta}^{\left(k_{\alpha+\beta}\right)}
e_{\beta}^{\left(k_\beta\right)} \right) := 
 \left(k_\alpha + k_{\alpha+\beta}\right)\alpha +
 \left(k_{\alpha+\beta} + k_{\beta}\right)\beta. 
$$
We also define the norm $\norm$ on $S$ by 
$$
\norm\left( k_\alpha\alpha + k_\beta\beta \right) := k_\alpha+k_\beta 
$$
 and will denote the composition of $\deg$ with $\norm$ by $\Deg$.

\section{Big Gr\"obner basis}
\label{bigbasis}
In this section we describe a 
Gr\"obner basis of the algebra $\U_3^+(\K)$ with respect to the generating set
$X =\left\{\,  e_{\alpha}^{\left(k_\alpha\right)}, 
e_{\alpha+\beta}^{\left(k_{\alpha+\beta}\right)}, 
e_{\beta}^{\left(k_\beta\right)}\,\middle|\, k_{\alpha}, k_{\alpha+\beta}, k_{\beta}\in \N \right\}$. 
Let $Y=\left\{ e_\alpha,e_{\alpha+\beta}, e_\beta \right\}$. We order $Y^*$ by
Deg-lexicographical ordering induced by the ordering
$$
e_{\alpha}<e_{\alpha+\beta}<e_\beta
$$
on $Y$. We have the map $\phi\colon X^*\to Y^*$ of free monoids induced by 
\begin{align*}
	e_{\alpha}^{\left( k \right)}& \mapsto e_{\alpha}^k & 
	e_{\alpha+\beta}^{\left( k \right)}& \mapsto e_{\alpha+\beta}^k & 
	e_{\beta}^{\left( k \right)}& \mapsto e_{\beta}^k. 
\end{align*}
We define the ordering $\ll$ on $X^*$ as follows.  If
$\phi\left( u \right) < \phi\left( v \right)$, then $u\ll v$. If $\phi\left( u
\right) = \phi\left( v \right)$ and the length of $u$ is less then the length of
$v$, then $u\ll v$. 
If $\phi\left( u \right) = \phi\left( v \right)$ and both words $u$ and
$v\in X^*$
have the same length, then we compare them lexicographically with respect to the
ordering
$$
e_{\alpha}< e_{\alpha+\beta}<e_{\beta}< e_{\alpha}^{\left( 2 \right)} <
e_{\alpha+\beta}^{\left( 2 \right)} < e_{\beta}^{\left( 2 \right)}< \dots <
e_{\alpha}^{\left( k \right)}< e_{\alpha+\beta}^{\left( k \right)}<
e_{\beta}^{\left( k \right)}<\dots
$$
on $X$. 
Since $\Deg$-lexicographical ordering on $Y^*$ is terminating and every fiber of
$\phi$ is finite, it follows that also the ordering $\ll$
on $X^*$ is terminating. It is also easy to see that $\ll$ is monomial. In fact,
let $u$, $v$, $w\in X^*$. Then $\phi\left( u \right)<\phi\left( v \right)$
implies $\phi\left( uv \right)<\phi\left( vw \right)$; 
if $\phi\left( u \right) = \phi\left( v \right)$ and the length of $u$ is less
then the length of $v$, then $\phi\left( uw \right) = \phi\left( vw
\right)$ and the length of $uw$ is less then the length of $vw$; if $\phi\left(
u \right) = \phi\left( v \right)$, $u$ and $v$ have the same length, and
$u<v$ with respect to the lexicographical ordering, then $\phi\left( uw
\right) = \phi\left( vw \right)$, $uw$ and $vw$ have the same length, and
$uw$ is less then $vw$ with respect to the lexicographical ordering. Thus
$u\ll v$ implies $uw\ll vw$. The stability with respect to the left
multiplication is verified analogous.

\begin{theorem}
Let $X$ and the ordering on $X$ be as above. Then the following set of rewriting rules is complete:
\begin{align}
	\label{alpha}
	e_{\alpha}^{\left( k \right)} e_{\alpha}^{\left( l \right)} & \to {k+l \choose k} e_{\alpha}^{\left( k+l \right)}\\
	\label{alphabeta}
	e_{\alpha+\beta}^{\left( k \right)} e_{\alpha+\beta}^{\left( l \right)} & \to {k+l \choose k} e_{\alpha+\beta}^{\left( k+l \right)}\\
	\label{beta}
	e_{\beta}^{\left( k \right)} e_{\beta}^{\left( l \right)} & \to {k+l \choose k} e_{\beta}^{\left( k+l \right)}\\
	\label{alphabetaalpha}
	e_{\alpha+\beta}^{\left( k \right)} e_{\alpha}^{\left( l \right)} &
	\mapsto e_{\alpha}^{\left( l \right)} e_{\alpha+\beta}^{\left( k
	\right)}\\
	\label{betaalpha}
	e_{\beta}^{\left( k \right)} e_{\alpha}^{\left( l \right)} & \mapsto
	\sum_{j=0}^{min\left( k,l \right)} \left( -1 \right)^{j}
	e_{\alpha}^{\left( l-j \right)} e_{\alpha+\beta}^{\left( j
	\right)} e_{\beta}^{\left( k-j \right)}\\
	\label{betaalphabeta}
	e_{\beta}^{\left( k \right)} e_{\alpha+\beta}^{\left( l \right)} &
	\mapsto e_{\alpha+\beta}^{\left( l \right)} e_{\beta}^{ \left(
	k \right)
	},
\end{align} 
where $k$, $l\in  \N$. 
	\label{big}
\end{theorem}
\begin{proof}
	It is clear that the set 
	$$
	B = \left\{ e_{\alpha}^{\left(k_\alpha\right)}
e_{\alpha+\beta}^{\left(k_{\alpha+\beta}\right)}
e_{\beta}^{\left(k_\beta\right)}\middle| k_{\alpha}, k_{\alpha+\beta}, k_{\beta}\in \N
\right\}
	$$
	is the set of non-reducible words with respect to the given rewriting
	system. By definition, the natural image of $B$ in
	$\U_3^+(\K)$ is a basis of $U_3^+(\K)$. Therefore, it is enough
	to check that for every rule the left hand side and the right hand side
	are equal in $\U_3^+(\K)$. 
	This is obvious for \eqref{alpha}, \eqref{alphabeta}, \eqref{beta},
	\eqref{alphabetaalpha}, \eqref{betaalphabeta}. Thus we have only to
	check the claim for \eqref{betaalpha}. 
	We have to prove the equality
	$$
e_{\beta}^{\left( k \right)} e_{\alpha}^{\left( l \right)} =  
	\sum_{j=0}^{min\left( k,l \right)} \left( -1 \right)^{j}
	e_{\alpha}^{\left( l-j \right)} e_{\alpha+\beta}^{\left( j
	\right)} e_{\beta}^{\left( k-j \right)} 
		$$
	in $\U_3^+(\K)$. Clearly it is enough to prove the same equality in
	$\U_3^+(\Z)$ and,  therefore in $\U_3^+(\C)$. We will do this by
	induction on the minimum of  $k$ and $l$. The case $min(k,l)=1$ splits
	into two cases $k=1$ and $l=1$. The case $k=1$, we prove by induction on
	$l$. For $k=l=1$ we have
	$$
	e_{\beta} e_{\alpha}= e_{\alpha}e_{\beta} - e_{\alpha+ \beta}.
	$$
	Suppose we have proved equality for $k=1$ and $l\le l_0$. Let us check
	it for $l=l_0+1$. 
	\begin{align*}
		e_{\beta}e_{\alpha}^{(l)} &=\frac{1}{l} 
		e_{\beta}e_{\alpha}^{(l-1)}e_{\alpha} \mbox{\ \ \ \ \ \ \ \ \ \     induction assumption} \\
		&= \frac{1}{l} \left(e_{\alpha}^{(l-1)}e_{\beta}-
		e_{\alpha}^{(l-2)}e_{\alpha+\beta}\right)e_{\alpha} \\
		&= \frac{1}{l}\left( e_{\alpha}^{(l-1)}e_{\alpha}e_{\beta} -
		e_{\alpha}^{(l-1)} e_{\alpha+\beta} -
		e_{\alpha}^{(l-2)}e_{\alpha}e_{\alpha+\beta}\right) \\
		&= e_{\alpha}^{(l)} e_{\beta} - \frac{1}{l}\left( 1 + l-1)
		\right)e_{\alpha}^{(r-1)} e_{\alpha+\beta} \\
&= e_{\alpha}^{(l)} e_{\beta} - e_{\alpha}^{(l-1)} e_{\alpha+\beta}. 
	\end{align*}
	Now we prove the equality in the case $l=1$ and $k\ge 2$. Suppose it is
	proved for all $k\le k_0$. Let us show it for $k=k_0+1$. We have
	\begin{align*}
		e_{\beta}^{(k)} e_{\alpha} & =
		\frac{1}{k}e_{\beta}e_{\beta}^{(k-1)}e_{\alpha} \\
		&= \frac{1}{k} \left( e_{\beta} e_{\alpha} e_{\beta}^{\left( k-1
		\right)}- e_{\beta} e_{\alpha+\beta}e_{\beta}^{(k-2)}  \right)\\
		&= \frac{1}{k} \left( e_{\alpha}e_{\beta}e_{\beta}^{(k-1)}
		-e_{\alpha+\beta}e_{\beta}^{(k-1)} - e_{\alpha+\beta}e_{\beta}e_{\beta}^{(k-2)}  \right)\\
		&= e_{\alpha}e_{\beta}^{(k)} - e_{\alpha+\beta} e_{\beta}^{(k-1)}.
	\end{align*}
Suppose we have prove equality for all $k$ and $l$ such that $min(k,r)\le m_0$.
Let us prove it for $min(k,r)= m_0+1$. There are two cases $k=m_0+1$ and
$l=m_0+1$. As the computations are very similar we will treat only the first
case. 
\begin{align*}
	e_{\beta}^{(k)}e_{\alpha}^{(l)}&=\frac{1}{k}e_{\beta}e_{\beta}^{(k-1)}
	e_{\alpha}^{(l)}\\
	&= \frac{1}{k}\sum_{s=0}^{k-1}\left( -1 \right)^s e_{\beta}e_{\alpha}^{(l-s)}
	e_{\alpha+\beta}^{(s)}e_{\beta}^{k-1-s} \\
	&= \frac{1}{k} \sum_{s=0}^{k-1}\left( -1 \right)^s \left( e_{\alpha}^{(l-s)} e_{\alpha+\beta}^{(s)} e_{\beta}
	e_{\beta}^{(k-1-s)} - e_{\alpha}^{(l-s-1)} e_{\alpha+\beta}e_{\alpha+\beta}^{(s)}
	e_{\beta}^{(k-1-s)}\right) \\
	&= \frac{1}{k} \sum_{s=0}^k  \left( (-1 \right)^s (k-s) -\left( -1
	\right)^{s-1} s) e_{\alpha}^{(l-s)} e_{\alpha+\beta}^{(s)}
	e_{\beta}^{(k)}= \sum_{s=0}^k\left( -1 \right)^s e_{\alpha}^{(l-s)} e_{\alpha+\beta}^{(s)}
	e_{\beta}^{(k)}.
\end{align*}
\end{proof}

\begin{corollary}
	\label{subalgebra}
	Let $p$ be a characteristic of the field $\K$ and $m\ge 0$. Then the
	linear span $\U_3^{m}(\K)$ of the set
	$$
	B'= \left\{\, e_{\alpha}^{\left( k_\alpha \right)}
	e_{\alpha+\beta}^{\left( k_{\alpha+\beta} \right)}
	e_{\beta}^{\left(k_\beta\right)}
	\,\middle|\,  k_{\alpha}, k_{\alpha+\beta}, k_{\beta}\le p^m-1 \right\}
	$$
	is a subalgebra of $\U_3^+(\K)$
\end{corollary}
\begin{proof}
We claim that $\U_3^m(\K)$ is the subalgebra $A$  of $\U_3^+(\K)$ generated
by the set 
$$
X'= \left\{\, e_{\alpha}^{(k)},e_{\alpha+\beta}^{(k)},e_{\beta}^{(k)}   \,\middle|\,  k\le p^m-1 \right\}. 
$$
It is enough to show that the set $B'$ is a basis of $A$.  
Let $R$ be rewriting system defined in Theorem~\ref{big}. We claim
that $R(X')$ is complete. To prove this we apply
Proposition~\ref{subsystem}. It is obvious for the rules
\eqref{alphabetaalpha}, \eqref{betaalpha}, and \eqref{betaalphabeta}, that if the left hand side
is an element of $\left( X' \right)^*$, then all the monomials on the
right hand side are also elements of $\left( X' \right)^*$. Moreover,
if $k+l\le p^m-1$ then the same is true for the rewriting rules \eqref{alpha},
\eqref{alphabeta}, and \eqref{beta}.  
 Suppose $k$, $l\le p^m-1$ and $k+l\ge p^m$. Then
${k+l \choose k}=0$ in $\K$. In fact,  the degree of
$p$ in the prime decomposition of $n!$ is given by the formula
$$
\sum_{j=0}^{\infty}\left[ \frac{n}{p^j} \right].
$$
Therefore, the degree of $p$ in the prime decomposition of ${k+l
\choose k}$ is 
\begin{align*}
 \sum_{j=0}^{\infty}\left( \left[ \frac{k+l}{p^j} \right]-\left[
 \frac{k}{p^j} 
 \right]-\left[ \frac{l}{p^j} \right] \right) &=
 \left[ \frac{k+l}{p^m} \right] + \sum_{j=0}^{l-1}
\left( \left[ \frac{k+l}{p^j} \right]-\left[
 \frac{k}{p^j} 
 \right]-\left[ \frac{l}{p^j} \right] \right)\\
 &\ge \left[ \frac{k+l}{p^m} \right]=1>0.
\end{align*}
Therefore, for the rules  \eqref{alpha},
\eqref{alphabeta}, \eqref{beta} and $k+l\ge p^m$, we get
\begin{align*}
	e_{\alpha}^{(k)}e_{\alpha}^{(l)}& \to 0&
	e_{\alpha+\beta}^{(k)}e_{\alpha+\beta}^{(l)}& \to 0&
	e_{\beta}^{(k)}e_{\beta}^{(l)}& \to 0.
\end{align*}
	This shows that $R(X')$ is complete. Now, it is obvious that $B'$ is
the set of non-reducible monomials in the alphabet $X'$ with respect to
the rewriting system $R(X')$. This shows that $B'$ is a basis of the
algebra $A'$.
\end{proof}
\section{Small Gr\"obner basis}
\label{smallbasis}
The Gr\"obner basis obtained in the previous section is not convenient
for the construction of minimal projective resolution of $\K$, since
the Anick resolution is much closer to the minimal resolution, if the
chosen generating set is minimal. 

Denote
$e_{\alpha}^{\left( p^k \right)}$ by $a_k$ and $e_{\beta}^{\left( p^k
\right)}$ by $b_k$. 
\begin{theorem}
	\label{genset}
	For any $m\in \N_0$
	the set $Z_m := \left\{\, a_l,b_l \,\middle|\, l\le m-1\right\}$ generates
	the algebra $\U_3^m\left( \K \right)$. And, therefore, the set
	$Z:= \left\{\, a_l, b_l \,\middle|\, l\in \N_0 \right\}$ generates
	$\U_3^+\left( \K \right)$. 
\end{theorem}
\begin{proof}
	We know that $U_3^m\left( \K \right)$ is generated by the elements
	$e_{\alpha}^{\left( k \right)}$, $e_{\alpha+\beta}^{\left( k
	\right)}$, $e_{\beta}^{\left( k \right)}$, $k\le p^m-1 $. Thus it is enough
	to show that these elements can be written as linear combination of
	monomials in $a_l$, $b_l$, $l\le m-1$. 

Suppose 
$k=k_lp^{m-1}+k_{l-1}p^{m-2}+\dots + k_0$ with $0\le k_s\le p-1$. Then it follows
from the Lucas' theorem \cite[(137)]{lucas} and \eqref{alpha},
\eqref{alphabeta}, \eqref{beta}, that 
\begin{align*}
	e_{\alpha}^{\left( k \right)}& = \prod_{s=0}^{m-1} e_{\alpha}^{\left( k_sp^s
	\right)}, &
	e_{\alpha+\beta}^{\left( k \right)}& = \prod_{s=0}^{m-1}
	e_{\alpha+\beta}^{\left( k_sp^s \right)}, &
	e_{\beta}^{\left( k \right)}& = \prod_{s=0}^{m-1} e_{\beta}^{\left( k_sp^s
	\right)}.
\end{align*}
Now, for any $0\le k_s\le p-1$ the integer $k_s!$ is invertible in $\K$.
Therefore 
\begin{align*}
	e_{\alpha}^{\left( k_sp^s \right)} & = \frac1{k_s!}\left( e_{\alpha}^{\left( p^s \right)} \right)^{k_s}, &
	e_{\alpha+\beta}^{\left( k_sp^s \right)} & = \frac1{k_s!}\left( e_{\alpha+\beta}^{\left( p^s \right)} \right)^{k_s}, &
	e_{\beta}^{\left( k_sp^s \right)} & = \frac1{k_s!}\left( e_{\beta}^{\left( p^s \right)} \right)^{k_s}. 
\end{align*}
Thus the algebra $\U_3^m\left( \K \right)$ is generated by the elements
$a_l$, $b_l$, and $e_{\alpha+\beta}^{p^l}$, $l\le m-1$. 

From \eqref{betaalpha}, it
follows
that
$$
e_{\alpha+\beta}^{\left( p^l \right)} = \left( -1 \right)^{p^l} \left(
e_{\beta}^{\left( p^l \right)}
 e_{\alpha}^{\left( p^l \right)} - \sum_{j=0}^{p^l-1}\left( -1
\right)^j e_{\alpha}^{\left( p^l-j \right)} e_{\alpha+\beta}^{\left( j
\right)} e_{\beta}^{\left( p^l-j \right)}\right).
$$
From this equality by
recursion on  $l$, it follows that
$e_{\alpha+\beta}^{\left( p^l \right)}$ can be written as a linear combination
of monomials in $a_s$, $b_s$ with $s\le l$. 
\end{proof}
%
%
We will consider $\Deg$-lexicographical order on $Z_m$
that corresponds to the ordering
$$
a_0<b_0<a_1<b_1<\dots<a_m<b_m.
$$
on $Z_m$. To establish the Gr\"obner basis of $\U_3^m\left(
\K \right)$ for the generating set $Z_m$ with respect to the above
ordering, we  prove some equalities between the elements $a_k$,
$b_k$, $k\in \N_0$. 

\begin{proposition}
	\label{squares}
	For any $k$ we have $a_k^p = b_k^p =0$. 
\end{proposition}
\begin{proof}
	We know that $a_k$ is an element of the subalgebra $\U_3^{k+1}\left(
	\K \right)$, and that $a_k^p$ is a linear multiple of $a_{k+1}$. Since
	$a_{k+1}\not\in \U_3^{k+1}\left( \K \right)$, it follows that the
	coefficient of multiplication is zero, and therefore $a_k^p=0$. The
	claim $b_k^p=0$ is proved in the same way. 
\end{proof}
\begin{proposition}
	\label{commute}
	For any $l$ and $k$ elements $a_l$ and $a_k$ commutes. Similarly
	$b_l$ and $b_k$.
\end{proposition}
\begin{proof}
	Obvious.
\end{proof}
\begin{proposition}
	\label{skew}
	For any $l>k$ we have
	\begin{align}
		\label{skewone}
		a_lb_k& - b_ka_l +\left( -1 \right)^{l-k} a_k^{p-1}b_ka_ka_{k+1}^{p-1}\dots
		a_{l-1}^{p-1} =0\\
		\label{skewtwo}
		b_la_k&- a_kb_l -\left( -1 \right)^{l-k} b_ka_kb_k^{p-1}b_{k+1}^{p-1}\dots
		b_{l-1}^{p-1}=0
		\end{align}
		in $\U_3^+(\K)$.
\end{proposition}
\begin{proof}
	First we note, that $a_k^{p-1} = -e_{\alpha}^{\left( (p-1)p^k \right)}$. In
	fact, $a_k^{p-1} = \left( p-1 \right)! e_{\alpha}^{\left( (p-1)p^k
	\right)}$. Now $\left( p-1 \right)!$ is the product of all elements in
	$\F_p^*$.  
	The elements of $\F_p^*\setminus\left\{ 1,-1 \right\}$ can be grouped in
	pairs $\left\{ \lambda, \lambda^{-1} \right\}$ with
	$\lambda\not=\lambda^{-1}$. Therefore the product $\left( p-1
	\right)!$ equals to $1\cdot -1 = -1$. 

	 By \eqref{betaalpha}, we get
	 \begin{align*}
		 b_k a_l &= e_{\beta}^{\left( p^k \right)} e_{\alpha}^{\left(
		 p^l \right)} = e_{\alpha}^{\left( p^l \right)}
		 e_{\beta}^{\left( p^k \right)} + \sum_{j=1}^{p^k} \left( -1
		 \right)^j e_{\alpha}^{\left( p^l-j \right)}
		 e_{\alpha+\beta}^{\left( j \right)} e_{\beta}^{\left( p^k-j
		 \right)}\\
		 &= a_l b_k + \sum_{j=1}^{p^k}\left( -1 \right)^j
		 e_{\alpha}^{\left( p^l-j \right)} e_{\alpha+\beta}^{\left( j
		 \right)} e_{\beta}^{\left( p^k-j \right)}. 
	 \end{align*}
	 For $1\le  p^k$, we have
	 \begin{align} 
		 \label{pl}
		 p^l-j &= \left( p-1 \right)p^{l-1} + \left( p-1
		 \right)p^{l-2} + \dots + \left( p-1 \right)p^k
	 	 +\left( p^k-j \right),
	 \end{align}
	 where $p^k-j\le p^k-1$. 
	 Therefore from \eqref{alpha} and the Lucas' theorem, it follows that
	 $e_{\alpha}^{\left( p^l-j \right)} = e_{\alpha}^{\left( \left( p-1
	 \right)p^k\right)}e_{\alpha}^{\left( p^l-p^{k+1}+p^k-j
	 \right)} $. Moreover, ${p^l \choose (p-1)p^k} = 0$ and
	 $e_{\alpha}^{\left( (p-1)p^k \right)} e_{\alpha}^{\left(
	 p^l-p^{k+1}-p^k \right)}
	 =0$. Therefore
	 \begin{align*}
		 \sum_{j=1}^{p^k} \left( -1 \right)^j e_{\alpha}^{\left( p^l-j
		 \right)} e_{\alpha+\beta}^{\left( j \right)}
		 e_{\beta}^{\left( p^k-j \right)} &= e_{\alpha}^{\left(
		 (p-1)p^k\right)} \sum_{j=0}^{p^k} \left( -1 \right)^j
		 e_{\alpha}^{\left( p^l -p^{k+1} + p^k -j 
		 \right)}e_{\alpha+\beta}^{\left( j \right)}
		 e_{\beta}^{\left( p^k-j \right)}\\
		 &= e_{\alpha}^{\left( (p-1)p^k  \right)} 
		 \left( e_{\beta}^{\left( p^k \right)}e_{\alpha}^{\left(
		 p^l-p^{k+1} +p^k
		 \right)} \right), 
	 \end{align*}
	 where in the last step we used \eqref{alphabeta}. 
	 Now
	 $$
	 p^l-p^{k+1} + p^k = \left( p-1 \right) p^{l-1} + \dots + \left( p-1
	 \right)p^{k+1} + p^k.
	 $$
	 Therefore by the Lucas' theorem and \eqref{alpha}
	 \begin{align*} 
	 e_{\alpha}^{\left(
 	 p^l-p^{k+1} +p^k
	 \right)}&= e_{\alpha}^{\left( p^k \right)}
 e_{\alpha}^{\left( (p-1)p^{k+1} \right)} \dots e_{\alpha}^{\left((p-1)p^{l-1}\right)}
 \\& =\left( -1 \right)^{l-k-1}a_ka_{k+1}^{p-1}\dots a_{l-1}^{p-1}.
 	\end{align*}
Finally we get
$$
b_ka_l = a_lb_k + (-1)^{l-k}a_k^{p-1}b_ka_ka_{k+1}^{p-1}\dots a_{l-1}^{p-1}. 
$$
Now we prove \eqref{skewtwo}. We have by \eqref{betaalpha}
\begin{align*}
	b_la_k &= e_{\beta}^{\left( p^l \right)} e_{\alpha}^{\left( p^k
	\right)} = e_{\alpha}^{\left( p^k \right)} e_{\beta}^{\left( p^l
	\right)} + \sum_{j=1}^{p^k}
	\left( -1 \right)^j e_{\alpha}^{\left( p^k-j \right)}
	e_{\alpha+\beta}^{\left( j \right)} e_{\beta}^{\left( p^l-j \right)}.
\end{align*} 
From \eqref{pl}, the Lucas' theorem and \eqref{beta}, we get
$$
e_{\beta}^{\left( p^l-j \right)} = e_{\beta}^{\left( p^k-j
\right)}e_{\beta}^{\left( p-1 \right)p^k}\dots e_{\beta}^{\left( p-1
\right)p^{l-1}}.
$$
Taking into the account that $b_s^{p-1} = - e_{\beta}^{\left( (p-1)p^s
\right)}$, for all $s\in \N_0$, and $$e_{\beta}^{\left( p^l-p^{k+1}+p^k
\right)}e_{\beta}^{\left( p^{k+1}-p^k \right)} = 0,$$  we get
\begin{align*}
	\sum_{j=1}^{p^k} \left( -1 \right)^j e_{\alpha}^{\left( p^k-j
	\right)}e_{\alpha+\beta}^{\left( j \right)} e_{\beta}^{\left( p^l-j
	\right)}  &= \left( \sum_{j=0}^{p^k} e_{\alpha}^{\left( p^k-j
	\right)}e_{\alpha+\beta}^{\left( j \right)}e_{\beta}^{\left( p^k-j
	\right)} \right)\left( -1 \right)^{l-k}b_{k}^{p-1}\dots
	b_{l-1}^{p-1}\\
	&= \left( -1 \right)^{l-k} b_ka_kb_k^{p-1}\dots b_{l-1}^{p-1}
\end{align*}
and \eqref{skewtwo} follows. 
\end{proof}

\begin{proposition}
	\label{braid}
For all $k\in \N_0$, we have $\left( b_ka_k \right)^p=\left( a_kb_k \right)^p$. 
\end{proposition}
\begin{proof}
We have
$$
\left( b_ka_k \right)^p = \left( a_kb_k +
\sum_{j=1}^{p^k-1}e_{\alpha}^{\left( p^k-j \right)}e_{\alpha+\beta}^{\left( j
\right)}e_{\beta}^{\left( p^k-j \right)} +\left( -1 \right)^{p^k} e_{\alpha+\beta}^{\left( p^k
\right)} \right)^p.
$$
Let $1\le j\le p^k-1$ and $1\le s\le p^k-1$. Then
\begin{multline*}
	e_{\alpha}^{\left( p^k-j \right)}e_{\alpha+\beta}^{\left( j
	\right)}e_{\beta}^{\left( p^k-j \right)} e_{\alpha}^{\left( p^k-s
	\right)} e_{\alpha+\beta}^{\left( s \right)} e_{\beta}^{\left( p^k-s
	\right)} =\\= \sum_{r=0}^{min\left( p^k-j,p^k-s \right)} \left( -1
	\right)^r e_{\alpha}^{\left( p^k-j \right)}e_{\alpha}^{\left( p^k-s-r
	\right)} e_{\alpha+\beta}^{\left( j \right)}e_{\alpha+\beta}^{\left(
	r \right)}e_{\alpha+\beta}^{\left( s \right)} e_{\beta}^{\left( p^k-j-r
	\right)}e_{\beta}^{\left( p^k-s \right)}.
\end{multline*}
Every monom on the right hand side of the last formula has the form
$$
e_{\alpha}^{\left( k_1 \right)} e_{\alpha}^{\left( k_2
\right)}
e_{\alpha+\beta}^{\left(k_3\right)}e_{\alpha+\beta}^{\left(k_4\right)}e_{\alpha+\beta}^{\left(k_5\right)}e_{\beta}^{\left(k_6\right)}e_{\beta}^{\left(k_7\right)},
$$
with $0\le k_i\le p^k-1$ for all $1\le i\le 7$. In particular, every such monom
is an element of $\U_3^k\left( \K \right)$. Moreover, we have equations
$$
\begin{array}{c@{\,}c@{\,}c@{\,}c@{\,}c@{\,}c@{\,}c@{\,}c@{\,}c@{\,}c@{\,}c@{\,}c@{\,}c@{\,}c@{\,}c@{\,}c}
	k_1 & + & & & k_3 & & & & & & & & & = & p^k\\
	    &   & k_2 & + & & & k_4 & + & k_5 & & & & &= & p^k\\
	   & & & & k_3 &+& k_4 & + & & & k_6 & & & = & p^k\\
	   & & & & & & & & k_5 & + & & & k_7 & = & p^k. 
\end{array}
$$
Thus $k_1+ k_2 + 2\left( k_3+k_4+k_5 \right) + k_6+k_7 = 4p^k$. Thereforeat
least one of
the following inequalities holds 
\begin{align*}
	k_1+ k_2 &\ge p^k\\
	k_3 + k_4 + k_5&\ge p^k\\
	k_6 + k_7 & \ge p^k.
\end{align*}
Since the elements with divided power greater or equal then $p^k$ do not lie in
$\U_3^k\left( \K \right)$, we get that
one of the products
\begin{align*}
	e_{\alpha}^{\left( k_1 \right)}e_{\alpha}^{\left( k_2 \right)} &&
	e_{\alpha+\beta}^{\left( k_3 \right)}e_{\alpha+\beta}^{\left( k_4 \right)} e_{\alpha+\beta}^{\left( k_5 \right)}
	&& e_{\beta}^{\left( k_6
	\right)}e_{\beta}^{\left( k_7 \right)}
\end{align*}
is zero. This shows that
$$
	e_{\alpha}^{\left( p^k-j \right)}e_{\alpha+\beta}^{\left( j
	\right)}e_{\beta}^{\left( p^k-j \right)} e_{\alpha}^{\left( p^k-s
	\right)} e_{\alpha+\beta}^{\left( s \right)} e_{\beta}^{\left( p^k-s
	\right)}  = 0.
$$
In particular, any two elements in the sum $\sum_{j=1}^{p^k-1}
e_{\alpha}^{\left( p^k-j \right)} e_{\alpha+\beta}^{\left( j
\right)}e_{\beta}^{\left( p^k-j \right)}$ commute. 

Now consider for $1\le j\le p^{k}-1$ the product
\begin{align*}
	e_{\alpha}^{\left( p^k \right)}e_{\beta}^{\left( p^k
	\right)}e_{\alpha}^{\left( p^k-j \right)}e_{\alpha+\beta}^{\left( j
	\right)}e_{\beta}^{\left( p^k-j \right)}&= e_{\alpha}^{\left( p^k
	\right)} e_{\alpha}^{\left( p^k-j \right)}e_{\alpha+\beta}^{\left( j
	\right)}e_{\beta}^{\left( p^k \right)}e_{\beta}^{\left( p^k-j \right)}
	\\&\phantom{=}+  
	\sum_{r=1}^{p^k-j} \left( -1 \right)^r
	e_{\alpha}^{\left( p^k \right)}e_{\alpha}^{\left(p^k-j-r
	\right)}e_{\alpha+\beta}^{\left( r \right)}e_{\alpha+\beta}^{\left( j
	\right)}e_{\beta}^{\left( p^k-r \right)}e_{\beta}^{\left( p^k-j
	\right)}.
\end{align*}
Every monom on the right hand side for $1\le r\le p^k-1$ can be written in the form 
$$
e_{\alpha}^{\left( p^k \right)}e_{\alpha}^{\left( k_1
\right)}
e_{\alpha+\beta}^{\left(k_2\right)}e_{\alpha+\beta}^{\left(k_3\right)}e_{\beta}^{\left(k_4\right)}e_{\beta}^{\left(k_5\right)}, 
$$
where $1\le k_i\le p^k-1$ for $2\le i\le 5$. Moreover, we have
\begin{align*}
	k_2+k_4 &= p^k\\
	k_3 + k_5 &= p^k,
\end{align*}
which implies $k_2+k_3+k_4+k_5 = 2p^k$ and, therefore,  $k_2+k_3\ge p^k$ or
$k_4+k_5\ge p^k$. By the same consideration as above, we get
that $e_{\alpha+\beta}^{\left( k_2 \right)}e_{\alpha+\beta}^{\left( k_3
\right)} =0$ or $e_{\beta}^{\left( k_4 \right)}e_{\beta}^{\left( k_5
\right)} =0$, and thus
$$
e_{\alpha}^{\left( p^k \right)}e_{\alpha}^{\left( k_1
\right)}
e_{\alpha+\beta}^{\left(k_2\right)}e_{\alpha+\beta}^{\left(k_3\right)}e_{\beta}^{\left(k_4\right)}e_{\beta}^{\left(k_5\right)}
= 0. 
$$
Therefore 
\begin{align*}
	e_{\alpha}^{\left( p^k \right)}e_{\beta}^{\left( p^k
	\right)}e_{\alpha}^{\left( p^k-j \right)}e_{\alpha+\beta}^{\left( j
	\right)}e_{\beta}^{\left( p^k-j \right)}  = e_{\alpha}^{\left( p^k
	\right)} e_{\alpha}^{\left( p^k-j \right)}e_{\alpha+\beta}^{\left( j
	\right)}e_{\beta}^{\left( p^k \right)}e_{\beta}^{\left( p^k-j \right)} .
\end{align*}
Similarly, it can be shown that 
for $1\le j\le p^k-1$
$$
e_{\alpha}^{\left( p^k-j \right)} e_{\alpha+\beta}^{\left( j \right)}
e_{\beta}^{\left( p^k-j \right)} e_{\alpha}^{\left( p^k
\right)}e_{\beta}^{\left( p^k \right)} = e_{\alpha}^{\left( p^k-j
\right)}e_{\alpha}^{\left( p^k \right)}e_{\alpha+\beta}^{\left( j
\right)}e_{b}^{\left( p^k-j \right)}e_{\beta}^{\left( p^k \right)}.
$$
Thus $e_{\alpha}^{\left( p^k \right)}e_{\beta}^{\left( p^k \right)}$ commute
with every summand of $$\sum_{j=1}^{p^k-1}\left( -1 \right)^j
e_{\alpha}^{\left( p^k-j \right)}e_{\alpha+\beta}^{\left( j
\right)}e_{\beta}^{\left( p^k-j \right)}.$$ It is obvious, that
$e_{\alpha+\beta}^{\left( p^k \right)}$ commutes with 
$e_{\alpha}^{\left( p^k \right)}e_{\beta}^{\left( p^k \right)}$ and
with every summand of $\sum_{j=1}^{p^k-1}\left( -1 \right)^j
e_{\alpha}^{\left( p^k-j \right)}e_{\alpha+\beta}^{\left( j
\right)}e_{\beta}^{\left( p^k-j \right)}$.
Therefore
\begin{align*}
	&\left( a_kb_k + \sum_{j=1}^{p^k-1} e_{\alpha}^{\left( p^k-j
	\right)}e_{\alpha+\beta}^{\left( j \right)}e_{\beta}^{\left( p^k-j
	\right)} + (-1)^{p^k} e_{\alpha+\beta}^{\left( p^k \right)} \right)^p =
	\\&\phantom{a_kb_k+} = 
	\left( a_kb_k \right)^p + \sum_{j=1}^{p^k-1}\left( -1 \right)^{jp^k}
	\left( e_{\alpha}^{\left( p^k-j \right)}e_{\alpha+\beta}^{\left( j
	\right)}e_{\beta}^{\left( p^k-j \right)} \right)^p + \left(
	e_{\alpha+\beta}^{\left( p^k \right)}
	\right)^p = \left( a_kb_k \right)^p. 
\end{align*}
\end{proof}
\begin{proposition}
	If $\chr\K\ge 3$, then
	for any $k\in \N_0$, we have 
	\begin{align*}
		b_k^2a_k - 2b_ka_kb_k + a_kb_k^2 &= 0\\
		b_ka_k^2 - 2a_kb_ka_k + a_k^2 b_k & =0.
	\end{align*}
\end{proposition}
\begin{proof}
	We have
	\begin{align*}
		b_k^2 a_k &= {2p^k \choose p^k} e_{\beta}^{\left( p^k \right)} e_{\alpha}^{\left( p^k
		\right)} = 2 \sum_{j=0}^{p^k}
		e_{\alpha}^{\left( p^k-j \right)} e_{\alpha+\beta}^{\left( j
		\right)} e_{\beta}^{\left( 2p^k-j \right)}\\
		&= 2e_{\alpha}^{\left( p^k \right)} e_{\beta}^{\left( 2p^k
		\right)} + 2 \sum_{j=1}^{p^k}
		e_{\alpha}^{\left( p^k-j \right)} e_{\alpha+\beta}^{\left( j
		\right)} e_{\beta}^{\left( 2p^k-j \right)}\\
a_kb_k^2 + 2 \sum_{j=1}^{p^k}
		e_{\alpha}^{\left( p^k-j \right)} e_{\alpha+\beta}^{\left( j
		\right)} e_{\beta}^{\left( 2p^k-j \right)}
	\end{align*}
On the other hand
	\begin{align*}
		b_ka_kb_k &= e_{\beta}^{\left( p^k \right)}e_{\alpha}^{\left(
		p^k \right)}e_{\beta}^{\left( p^k \right)} =
		\sum_{j=0}^{p^k} e_{\alpha}^{\left( p^k-j
		\right)}e_{\alpha+\beta}^{\left( j \right)}e_{\beta}^{\left(
		p^k-j \right)} e_{\beta}^{\left( p^k \right)}.
		\end{align*}
		
		For $j=0$, we have $e_{\beta}^{\left( p^k-j \right)}
		e_{\beta}^{\left( p^k \right)} = b_k^2$. On the other
		hand, for $1\le j\le p^k$, we have $2p^k = p^k + \left( p^k-j
		\right)$, and $0\le p^k-j\le p^k-1$. Therefore, from the Lucas'
		theorem if follows that ${2p^k-j \choose p^k} = 1$, and
		$e_{\beta}^{\left( p^k-j \right)}e_{\beta}^{\left( p^k
		\right)} = e_{\beta}^{\left( 2p^k-j \right)}$. Thus
		$$
		b_ka_kb_k = a_kb_k^2 + \sum_{j=1}^{p^k}e_{\alpha}^{\left( p^k-j
		\right)}e_{\alpha+\beta}^{\left( j \right)} e_{\beta}^{\left(
		2p^k-j \right)}
		$$
		and 
		$$
		b_k^2a_k - 2b_ka_kb_k = a_kb_k^2.
		$$
	The second equality follows from the first one, after noticing that
	$e_{\alpha}^{\left( k \right)} \mapsto e_{\beta}^{\left( k
	\right)}$, $e_{\beta}^{\left( k \right)}\mapsto e_{\alpha}^{\left( k
	\right)}$, $e_{\alpha+\beta}^{\left( k \right)}\mapsto \left( -1
	\right)^k e_{\alpha+\beta}^{\left( k \right)}$ can be prolonged to an
	automorphism of $\U_3^+\left( \K \right)$.	
\end{proof}

	Denote by $\pi_m$ the natural projection $\K\left\langle Z^*_m\
	\right\rangle\to \U_3^m\left( \K \right)$.
\begin{proposition}
	\label{gbl}
	Suppose $\chr\K =p$. 
	The following set $G_m$ of elements in $\K\left\langle X_m^* \right\rangle$
	\begin{align}
		a_lb_k& - b_ka_l + (-1)^{l-k}a_k^{p-1}b_ka_ka_{k+1}^{p-1}\dots
		a_{l-1}^{p-1},& 0\le k<l\le m-1\\
		b_la_k&- a_kb_l -\left( -1 \right)^{l-k} b_ka_kb_kb_{k+1}\dots
		b_{l-1},&
		0\le k<l\le m-1\\
		a_la_k& + a_ka_l, & 0\le k<l\le m-1\\
		b_lb_k & + b_kb_l,& 0 \le k<l\le m-1\\
		(b_ka_k)^p&-  (a_kb_k)^p, & 0\le k\le m-1\\
		a_k^p ,& & 0\le k\le m-1\\
		b_k^p, & & 0 \le k\le m-1\\
		b_k^2a_k& - 2b_ka_kb_k + a_kb_k^2, & 0\le k\le m-1,\ p\ge 3,\\
		b_ka_k^2& - 2a_kb_ka_k + a_k^2b_k, & 0 \le k\le m-1, \ p\ge 3
	\end{align}
	is a reduced Gr\"obner basis of $\ker(\pi_m)$. 
\end{proposition}
\begin{proof}
	If follows from Propositions~\ref{squares},~\ref{commute},~\ref{skew},~\ref{braid}, that
	$G_m$ is a subset of $\ker(\pi)$. 
	Thus it is enough to show that the images of non-reducible monomials in
	$X_m^*$ give a basis of $\U^m_3(\K)$. Since the images of
	non-reducible monomials in 
	$X^*_m$ generate $\U^m_3(\K)$ as a vector space and
	$\U^m_3(\K)$ is finite dimensional, it is enough to show that
	the number of non-reducible monomials in $X_m^*$ with respect to $G_m$ is
	less or equal to the dimension of $U^m_3(\K)$. From
	Corollary~\ref{subalgebra} it follows that the dimension of
	$\U_3^m(\K)$ is $\left( p^{m} \right)^3= p^{3m}$. 

	 Let $t$ be a monomial non-reducible with
	respect to $G_m$. Since $t$ does not contain  submonomials $a_lb_k$, $b_la_k$, $a_la_k$,
	$b_lb_k$ for $0\le k<l\le m-1$, the indices of variables in $t$ weakly
	increase from the left to right. We denote by $t_k$ a submonomial of
	$t$ that consists from the all variables with index $k$. Then $t=t_0 t_1
	\dots t_{m-1}$. 
	
	For $0\le k\le m-1$, the monomial $t_k$ has the form
	$$
	a_k^{r_1}b_k^{s_1}\dots a_k^{r_n}b_k^{s_n}, 
	$$
	where $n\in \N$, $r_i$ and $s_i$ are non-zero natural numbers, except
	probably of $r_1$ and $s_n$. Since $t_k$ does not contain subnomials
	$b_k^2a_k$ and $b_ka_k^2$, we see that $s_1=r_2=\dots = s_{n-1} = r_n =
	1$. Thus $t_k$ has the form $a_k^{r_1}\left( b_ka_k
	\right)^{n-1}b_k^{s_n}$. As $t_k$  does not contain subwords $a_k^p$,
	$\left( b_ka_k \right)^p$, and $b_k^p$, we get that $r_1\le p-1$,
	$n-1\le p-1$, and $s_n\le p-1$. Thus the number of different
	possibilities for $t_k$ does not exceed $p^3$, and the number of
	different possibilities for $t$ does not exceed $\left( p^3
	\right)^{m} = p^{3m}$. 
\end{proof}
\begin{corollary}
	\label{gb}
	The set $G$ is a reduced Gr\"obner basis of
	$\ker(\pi)$, where $\pi$ is the natural projection $\K\left\langle Z^*
	\right\rangle\to \U_3^+(\K)$. 
\end{corollary}
\begin{proof}
	It is clear that $G\subset \ker(\pi)$. Denote by $R$ the rewriting
	system $\left\{\, r(p) \,\middle|\, p\in G \right\}$. It is enough to show that   
	any critical pair $(w,r_1,r_2)$, with $r_1$, $r_2\in R$ is reducible.
	For a given critical pair $(w,r_1,r_2)$ there is an $m\ge 0$, such that
	all monomials in $w$, $r_1$, $r_2$ lie in $Z_m$. By
	Proposition~\ref{gbl} the set $G_m$ is a Gr\"obner basis, therefore any
	critical pair $(w,r_1,r_2)$ with $w\in Z_m^*$, $r_1$, $r_2\in
	R_m=\left\{\, r(p) \,\middle|\, p\in G_m \right\}$ is reducible.
\end{proof}

\section{Anick resolution}
\label{Anickres}

The Anick resolution was introduced in~\cite{Anick}. Let $A$ be an algebra over
a field $\K$ and $\varepsilon\colon A\to \K$ a homomorphism of algebras. Let $X =
a_1$,\dots be a set of generators of $A$ and $G\subset \K\left\langle X^*
\right\rangle$ a reduced Gr\"obner basis with respect to a monomial ordering
$\le$ on
$X^*$. For this set of data Anick constructed a free resolution of $\K$ over
$A$, which is nowadays called Anick resolution. We will describe only the first
four steps of Anick's construction under additional assumption that
$\varepsilon(x)=0$ for all $x\in X$.

First we define sets $T_k$, $k=-1,0,1,2$, that will serve as bases of $A$-free
modules $P_k$. Denote by $T_{-1}$ the set $\left\{ e \right\}$ with one element
$e$ and by  $T_0$ the set $X$.   
The set $T_1$ is the set of all leading monomials in $G$.
 Denote by $\widetilde{T}_2$ the set of all possible overlaps of
 elements of $T_1$. Every element of $\widetilde{T}_2$ is a triple
 $(w,r_1,r_2)$. We say that an overlap $(w,r_1,r_2)$ is minimal if there is no
 overlap $(w',r'_1,r'_2)$ such that $w'$ is a subword of  $w$. Note that if an overlap
 $(w,r_1,r_2)$ is minimal then the rules $r_1$ and $r_2$ are uniquely determined
 by $w$. In fact, suppose that $(w,r_1,r_2)$ $(w,r'_1,r'_2)\in
 \widetilde{T}_2$. Then $w=m_1v=m'_1v'$. But this means that $m_1$ is a subword
 of $ m'_1$
 or $m'_1$ is a subword of $ m_1$. Since $G$ is a reduced Gr\"obner basis it follows that
 $r_1=r'_1$. Similarly  $r_2=r'_2$. 
 
 Define $T_2$ to be the set of all words $w$ in
 $X^*$ such that there exists a  minimal overlap $(w,r_1,r_2)$.  Denote for $k=-1$, $0$, $1$, $2$
 by $P_k$ the $A$-linear span of $T_k$. Let $M$ be the set of all
non-reducible monomials with respect to $G$. Then for $k=-1,0,1,2$ the set
$$
N_k = \left\{\, m.t \,\middle|\,  m\in M,\ t\in T_k \right\}
$$
is the basis of $P_k$ over $\K$.

The sets $N_k$ have a full ordering induced by the ordering $\le$ on $X^*$ via
the map $m.t\mapsto mt$. We define maps $\delta_n\colon P_n\to P_{n-1}$ and
$j_n\colon P_{n-1}\to P_n$ as follows
\begin{align*}
	\delta_0(m.x)& := NF(mx,G).e  \\ j_0(ux.e)& := u.x\\
	\delta_1(m.t)&:= NF(mt',G).x,\mbox{ where $t=t'x$}. 
\end{align*} 
Now let $m\in M$ and $x\in X$. Suppose there are  $u$, $v\in M$ such that
$m=uv$ and $vx\in T_1$. Then we define $j_1(m.x)= u.vx$. Otherwise we let
$j_1(m.x)=0$. Note that $j_1$ is well-defined as $m=uv=u'v'$ would imply that
$v\preceq v'$ 
or $v'\preceq v$ and therefore $vx\preceq v'x$ or $v'x\preceq v'x$. But since
$G$ is reduced Gr\"obner basis any two different  elements of $T_1$ are
incompatible with respect to $\preceq$ (in other words $T_1$ is an anti-chain in
the Anick's terminology). 

Let $w\in T_2$ be such that $w=m_1v=um_2$ with $m_1$, $m_2\in T_1$. Define
$\delta_2(m.w)= NF(mu,G).m_2$. 

Suppose $t\in T_1$ and $m\in M$. If $m=uv$ for some $u$, $v\in M$ such that
$vt\in T_2$ then we define $j_2(m.t)=u.vt$. Note that if such $u$ and $v$ exist
then they are unique as $G$ is a reduced Gr\"obner basis. If there is no
$u$ and $v$ with the above property then we let $j_2(m.t)=0$. 

Note, that if $A$ is an $S$-graded algebra, where $S$ is a monoid, then the maps
$j_n$ and $\delta_n$ are homomgeneous with respect to the induced grading on the
modules $P_k$.

Now we define homomorphisms of left $A$-modules $d_n\colon P_n\to P_{n-1}$ and
homomorphisms of $\K$-vector spaces $i_n\colon \ker(d_{n-1})\to P_n$ for $n=0,1,2$ by
induction. Since $d_n$ is a homomorphism of free $A$-modules it is
enough to define $d_n$ on the basis elements $.t$, where $t\in T_n$.
On the other hand $i_n$ is a homomorphism of $\K$ vector spaces,
moreover we do not have any convenient basis for $\ker(d_{n-1})$. We
will define $i_n$ by induction on the leading term of $f\in
\ker(d_{n-1})$.
\begin{align*}
	d_0 (.t)& := \delta_0(.t)\\
	i_0 (m.e)& := j_0(m.e)\\
	d_{n+1} (.t)&:= \delta_{n+1}(.t) - i_nd_n(\delta_{n+1}(.t)) \\
	i_n(f) &:= j_n(\lt(f)) + i_n(f-d_n(j_n(\lt(f)))).
\end{align*}
 Note that it is not obvious that $d_n$ and $i_n$ are
 well-defined. This a part of the claim of Proposition~\ref{Anick}.
The following proposition is proved in~\cite{Anick}. Note that Anick~\cite{Anick}
constructed modules $P_n$ and maps $d_n$ for all $n\in \N$. 
\begin{proposition}
	\label{Anick} The sequence of left $A$-modules 
	$$
	P_2 \stackrel{d_2}{\longrightarrow} P_1 \stackrel{d_1}{\longrightarrow}
	P_0 \stackrel{d_0}{\longrightarrow} P_{-1}   \stackrel{\varepsilon}{\longrightarrow}\K \to 0
	$$
	is an exact complex. 
	If $A$ is an $S$-graded algebra, then all the maps above are
	homogeneous. Moreover, the highest term in the expansion $d_n\left( .t
	\right)$ is $\delta_n\left( .t \right)$ for every $n$ and $t\in T_n$. 
\end{proposition}

Let $j<m$. We will consider the algebra $\U_3^{km}\left( \K \right)$ with the
generating set $X_{jm}$ and the Gr\"obner basis $G_{jm}$. The sets $T_1$ and
$T_2$ in this case are given by
\begin{align*}
	T_1 &= \left\{\, a_lb_k,\ b_la_k, a_la_k,\ b_lb_k  \,\middle|\, j\le
	k<l\le m-1 \right\}\\&\phantom{:=} \cup \left\{\, a_k^p,\ b_k^p, \left( b_ka_k
	\right)^p \,\middle|\, j\le k\le m-1 \right\}\\ &\phantom{:=} \cup \left\{\, b_k^2a_k,
	b_ka_k^2
	\,\middle|\, j\le k\le m-1, p\ge 3 \right\}
\end{align*}
\begin{align*}
	T_2 &= \left\{\, 
	\begin{matrix}
	a_ra_la_k, a_ra_lb_k, a_rb_la_k, a_rb_lb_k,\\
	b_ra_la_k,b_ra_lb_k, b_rb_la_k, b_rb_lb_k\end{matrix} \,\middle|\, j\le k<l<r\le m-1
	\right\} \\& \phantom{=}
	\cup \left\{\,\begin{matrix} a_la_k^p, a_lb_k^p, a_l\left( b_ka_k \right)^p, b_la_k^p,
	b_lb_k^p, b_l\left( b_ka_k \right)^p,\\[1ex] a_l^pa_k, a_l^pb_k, b_l^pa_k,
	b_l^pb_k, \left( b_la_l \right)^pa_k, \left( b_la_l
	\right)^pb_k\end{matrix}\,\middle|\,j\le k<l\le m-1 \right\}
	\\&\phantom{=}
	\cup \left\{\,\begin{matrix} a_lb_k^2a_k, a_lb_ka_k^2, b_lb_k^2a_k,
		b_lb_ka_k^2,\\[1ex]
		b_l^2a_la_k, b_l^2a_lb_k, b_la_l^2a_k, b_la_l^2b_k\end{matrix} \,\middle|\, j\le
	k<l\le m-1, p\ge 3 \right\}
	\\&\phantom{=}\cup 
	\left\{\, a_k^{p+1}, b_k^{p+1}, b_k\left( b_ka_k \right)^p, \left(
	b_ka_k
	\right)^pa_k, \left( b_ka_k \right)^{p+1}\,\middle|\, j\le k\le m-1
	\right\} \\&\phantom{=} \cup \left\{\, b_k^2a_k^2 \,\middle|\, j\le k\le
	m-1, p\ge 3 \right\}.
\end{align*}
In the tables below we list the values of $\deg$ for the elements of $T_1$ and
$T_2$.
$$
\begin{array}{cc|cc|cc}
w\in 	T_1 & \deg\left( w \right) & w\in T_1 & \deg\left( w \right) & w\in
T_1,\ p\ge 3 & \deg\left( w \right)\\
	\hline
	a_lb_k & p^l\alpha + p^k \beta & a_k^p & p^{k+1} \alpha & b_k^2a_k & p^k
	\alpha + 2p^k \beta \\
	b_l a_k & p^k \alpha + p^l \beta & b_k^p & p^{k+1}\beta & b_ka_k^2 &
	p^k\beta + 2p^k\alpha
	\\
	a_la_k & \left( p^l+p^k \right)\alpha & \left( b_ka_k \right)^p
	& p^{k+1}\alpha + p^{k+1}\beta\\
	b_lb_k & \left( p^l + p^k \right)\beta
\end{array}
$$
For $j\le k<l<r\le m-1$
$$
\begin{array}{cc|cc}
	w\in T_2 & \deg\left( w \right) & w\in T_2 & \deg\left( w \right)
\\
\hline
a_ra_la_k & \left( p^r+p^l + p^k \right)\alpha & b_ra_la_k & (p^l+ p^k)\alpha +
p^r \beta \\
a_ra_lb_k & \left( p^r + p^l \right)\alpha + p^k\beta & b_ra_lb_k & p^l\alpha +
\left(  p^r + p^k 
\right) \beta\\
a_rb_l a_k & \left( p^r + p^k  \right) \alpha + p^l \beta & b_rb_l a_k & p^k
\alpha +  \left( p^r + p^l 
 \right)\beta\\
 a_rb_lb_k & p^l \alpha + \left( p^l + p^k  \right) \beta & b_rb_lb_k & \left(
 p^r +p^l + p^k \right)\beta
\end{array}
$$
For $j\le k<l\le m-1$
$$
\begin{array}{cc|cc}
w\in T_2 & \deg\left( w \right) & w\in T_2 & \deg\left( w \right)
\\
\hline
a_la_k^p & \left( p^l + p^{k+1} \right)\alpha & a_l^pa_k & \left( p^{l+1} + p^k 
\right)\alpha \\
a_lb_k^p & p^l\alpha + p^{k+1}\beta & a_l^p b_k & \left( p^{l+1} \right)\alpha +
p^k \beta \\
a_l\left(b_ka_k \right)^p & \left( p^l + p^{k+1} \right)\alpha + p^{k+1}\beta &
b_l^p a_k & p^k \alpha + p^{l+1}\beta \\
b_l a_k^p & p^{k+1}\alpha + p^l \beta & b_l^p b_k & \left( p^{l+1}+ p^k
\right)\beta \\
b_lb_k^p & \left( p^l + p^{k+1} \right) \beta & \left( b_la_l \right)^pa_k &
\left( p^{l+1}+ p^k  \right)\alpha + p^{l+1}\beta \\
b_l\left( b_ka_k \right)^p &  p^{k+1}\alpha + \left( p^{k+1} + p^l \right)\beta
& \left( b_la_l \right)^p b_k & p^{l+1}\alpha + \left( p^{l+1} + p^k \right)
\beta
\end{array}
$$
For $j\le k\le m-1$
$$
\begin{array}{cc|cc}
w\in T_2 & \deg\left( w \right) & w\in T_2 & \deg\left( w \right)	
\\
\hline
a_k^{p+1} & \left( p^{k+1} + p^k \right) \alpha & b_k\left( b_ka_k \right)^p &
p^{k+1}\alpha + \left( p^{k+1} + p^k \right)\beta
\\ b_k^{p+1} & \left( p^{k+1} + p^k \right)\beta & 
\left( b_ka_k \right)^p a_k & \left( p^{k+1} + p^k \right)\alpha + p^{k+1}\beta
\\& & \left( b_ka_k
\right)^{p+1} & \left( p^{k+1} + p^k \right)\left( \alpha + \beta \right)
\end{array}
$$
For $p\ge 3$
$$
\begin{array}{cc|cc}
	w\in T_2 & \deg\left( w \right) & w\in T_2 & \deg\left( w \right)	
\\
\hline
a_lb_k^2a_k & \left( p^l + p^k \right)\alpha + 2p^k \beta & b_l^2 a_la_k &
\left( p^l + p^k \right)\alpha + 2p^l \beta \\
a_lb_ka_k^2 & \left( p^l + 2p^k \right)\alpha + p^k \beta & b_l^2 a_lb_k & p^l
\alpha + \left( 2p^l + p^k \right)\beta \\
b_lb_k^2a_k & p^k \alpha + \left( 2p^k + p^l \right)\beta & 
b_l a_l^2 a_k & \left( 2p^l + p^k \right)\alpha + p^l \beta\\
b_lb_ka_k^2 & 2p^k\alpha + \left( p^l + p^k \right)\beta & b_la_l^2b_k & 
2p^l\alpha + \left( p^l + p^k \right)\beta\\
b_k^2 a_k^2 & 2p^k\alpha + 2p^k \beta
\end{array}
$$
Analizing these tables we get
\begin{proposition}
	\label{matches} The set $W$ of pairs $\left( w_1,w_2 \right)\in T_1\times
	T_2$ such that $\deg\left( w_1 \right) = \deg\left( w_2 \right)$ is given by 
	\begin{multline*}	
		\left\{\,\begin{matrix}\left( a_lb_k, a_lb_{k-1}^p \right),
			\left( a_lb_k,
			a_lb_{k-1}^p \right), \left( \left( b_ka_k \right)^p,
			a_{k+1}b_k
			\right), \left( b_la_k, b_la_{k-1}^p \right),\\[1ex]
			\left( \left( b_ka_k
			\right)^p, b_{k+1}a_k \right), \left( b_lb_k,
			b_lb_{k-1}^p \right),
			\left( a_{l+1}a_k, a_l^pa_k \right), \left( a_{l+1}b_k,
		a_l^pb_k
		\right),\\[1ex] \left( b_{l+1}a_k, b_l^pa_k \right), \left(
		b_{l+1}b_k,
		b_l^pb_k \right), \left( a_{k+1}a_k, a_{k}^{p+1} \right), \left(
		b_{k+1}b_k, b_k^{p+1}
		\right)\end{matrix}  \, \right\}
		\\ \cup 
		\left\{\,\left( a_k^2, a_ka_{k-1}^2 \right), \left( b_k^2,
		b_kb_{k-1}^2
		\right), \left( a_{l+1}b_l,
		a_l\left( b_{l-1}a_{l-1} \right)^2, \left( b_{l+1}a_l, b_l\left(
		b_{l-1}a_{l-1} \right)^2\right)
		\right) \,\middle|\, p=2 \right\}.
	\end{multline*}
	\end{proposition}

\section{Minimal resolution}
\label{firststeps}
Let $\Gamma$ be a monoid, and $A$ an $\Gamma$-graded algebra over a field $\K$. For
$\gamma\in\Gamma$ we denote by $A[\gamma]$ the left $\Gamma$-graded $A$-module defined by
$$
A\left[ \gamma \right]_{\alpha} := \bigoplus_{\beta\in\Gamma\colon \beta\gamma=
\alpha}A_{\beta}. 
$$
The modules $A\left[ \gamma \right]$ are projective modules in the category of
$S$-graded $A$-modules, and every projective module is a direct summand of a
direct sum of modules of the form $A\left[ \gamma \right]$ ( cf.
\cite[Propositions~4.1, 5.1]{semiperfect}
). 

Let $M$ be a $\Gamma$-graded $A$-module. A submodule $N$ of $M$ is called
\emph{small} if for any submodule $T$ of $M$ such that $T+N = M$, we have
$T=M$. 
A \emph{projective cover} of a module $M$ is a projective object $P$ together
with an epimorphism $\psi\colon P\twoheadrightarrow M$ such that the kernel of
$\psi$ is a small subobject of $P$. It is well known fact that if projective
cover exists for the module $M$, the it is unique (cf. \cite[Theorem~5.1]{semiperfect}). 
It is proved in Proposition~5.2 of \cite{perfect}, that if $\Gamma$ is an
artinian ordered monoid such that the neutral element $0$ of $\Gamma$ is the minimal
element, and $A_0$ is a perfect ring, then every $\Gamma$-graded $A$-module
possess a projective cover. Note that the conditions of this criterion hold
for $\U_3^+\left( \K \right)$ considered as an $S$-graded algebra, where
$S$ is the free commutative monoid generated by $\alpha$ and $\beta$. In fact,
$S$ is artinian with respect to the lexicographical ordering, and $0$ is the
minimal element with respect to this ordering. Moreover, the $0$-th component of
$\U_3^+\left( \K \right)$ is isomorphic to $\K$ and thus is perfect. 

A projective resolution 
$$
\dots \to P_k \stackrel{d_{k}}{\longrightarrow}\dots \stackrel{d_2} P_1
\stackrel{d_{1}}{\longrightarrow} P_0 \stackrel{d_0}{\longrightarrow}
P_{-1} \stackrel{d_{-1}}{\longrightarrow} M \to 0
$$
of a $\Gamma$-graded $A$-module $M$ is called minimal if 
the kernel of $d_j$ is a small subobject of $P_j$, $j\ge -1$, or, equivalently,
if the image of $d_i$ is a small subobject of $P_{i+1}$ for $i\ge 0$. 
From the uniqueness of the projective cover, it follows that the minimal
resolution is unique up to isomorphism if it exists. Moreover, if every
$\Gamma$-graded $A$-module possess a  projective cover, then the existence of a
minimal projecitve resolution for every $\Gamma$-graded $A$-module can be proved
by induction. 

\begin{proposition}
	Let $\left( s_i \right)_{i\in I}$ be a family of elements in $S$. Denote
	by  $N_i$ the direct sum
	$\bigoplus_{s\in S, s\not=0} \U_3^+\left( \K \right)[s_i]_s$. Then
	$N\subset \bigoplus_{i\in I} \U_3^+\left( \K
	\right)[s_i] $ is a small submodule  if and only if $M\subset \bigoplus_{i\in I} N_i$.
\end{proposition}
We will denote the submodule  $\bigoplus_{i\in I} N_i$ of $P:= \bigoplus_{i\in I}
\U_3^+\left( \K \right)[s_i]$, with $N_i$ defined as above, by $\Rad\left( P
\right)$. 
\begin{corollary}
	The Anick resolution of $\U_3^+\left( \K \right)$ 
	with respect to the Gr\"obner basis $G$ gives two first steps of the
	minimal resolution for the trivial module $\K$. 
\end{corollary}
\begin{proof}
	We have to check that the images of $d_1$ and $d_0$ are small suobjects
	of $P_0$ and $P_{-1}$, respectively. For $d_0$, this is obvious, as
	$d\left( .x \right)= x.e$ for any $x\in X$, and $x.e \in \U_3^+\left(
	\K \right)_{\deg\left( x \right)}$, $\deg\left( x \right)\not=0$.  

	Now, let $w\in T_1$. We write $w$ in the form $u.x$, where $x\in X$. Then
	$$
	d_1\left( .w \right) = u.x -i\left( NF\left( w,G \right) \right).
	$$
	Since, $NF\left( w,G \right)$ does not contain the terms of length one,
	we see that $i\left( NF\left( w,G \right) \right)$ is an element of
	$\bigoplus_{x\in X}\bigoplus_{s\not=0} \U_3^+\left( \K \right)[\deg\left( y
	\right)]_{s}$. 
\end{proof}
Unfortunately, the third step of the Anick resolution, in our situation, does
not give the third step of the minimal resolution, since $P_1$ contains redudant
summands. To find out which summands of $P_1$ should be skipped, we will analyze
the differential $d_2\colon P_2\to P_1$. 

First of all, note that if $w\in T_2$ is such that $\deg\left( w \right)$ is
different from the degrees of elements in $T_1$, then 
$d_2\left( w \right)\in \Rad\left( P_1 \right)$. 

Thus, the only elements $w\in T_2$, for which it can happen that $d_2\left( .w
\right)\not\in\Rad\left( P_1 \right) $, are given by the second components of
the pairs in $W$ defined in Proposition~\ref{matches}.

Now, suppose $w\in T_2$ and $u_1$, \dots, $u_k\in T_1$ is the full set of
elements in $T_1$ such that $\deg\left( w \right) = \deg\left( u \right)$. Then
$d_2\left( .w \right)$ is an element of $\Rad\left( P_1 \right)$ if and only if
coefficients of $.u_1$, \dots, $.u_k$ in the expansion of $d_2\left( .w
\right)$ are zero. 
\begin{lemma}
	\label{zero}
	Suppose $w\in T_2$ and $u\in T_1$ are such that $\deg\left( w
	\right) = \deg\left( u \right)$ and $u<w$. Then the coefficient of
	$.u$ in the expansion of $d_2(.w)$ is zero.  
\end{lemma}
\begin{proof}
	Recall that the $\K$-basis of $P_1$ is given by $B:=\left\{\, v'.v''
	\,\middle|\, NF\left( v',G \right)=v', v''\in T_1 \right\}$. 
	This basis is ordered via the map $B\to X^*$, $v'.v''\mapsto v'v''$.
	Now, write $w$ in the form $v'v''$ with $v''\in T_2$ and $v'$
	non-reducible word in $X^*$. Then the maximal element in the expansion
	of $d_2\left( .w \right)$ is $\delta_2\left( .w \right) = v'.v''$. Since
	$w<u$, we have $v'.v''<u$. Therefore, the coefficient of $.u$
	 in the expansion of $d_2\left( .w \right)$ is zero. 
\end{proof}
\begin{corollary}
	The images of $.a_lb_k^{p-1}$, $.b_la_{k}^p$ for $l\ge k+2$, of
	$.a_la_k^{p}$, $.b_lb_k^{p}$, $.a_l^pa_k$, $.a_l^pb_k$, $.b_l^pa_k$,
	$.b_l^pb_k$ for $l\ge k+1$, of  $.a_k^{p+1}$, $.b_k^{p+1}$ for any
	$k$, and of $.a_ka_{k-1}^2$, $.b_kb_{k-1}^2$, $.a_k\left(
	b_{k-1}a_{k-1}
	\right)^2$, $.b_k\left( b_{k-1}a_{k-1} \right)^2$ for $p=2$ under
	$d_2$ are the elements $\Rad\left( P_1 \right)$. 
\end{corollary}
Note that for the elements $b_{k+1}a_{k}^p$ and $a_{k+1}b_{k}^p$ of $T_2$ there are two
elements in $T_1$ of the same degree, namely
$a_{k+1}b_{k+1}$ and $\left( b_ka_k \right)^p$. Since, $a_{k+1}b_{k+1}>
a_{k+1}b_{k}^p$, it follows from Lemma~\ref{zero} that the coefficient of
$.a_{k+1}b_{k+1}$ in the expansion of $d_2\left( .a_{k+1}b_k^{p} \right)$ is
zero. 
\begin{proposition}
	The coefficient of $.a_{k+1}b_{k+1}$ in the expansion of $d_2\left(
	.b_{k+1}a_k^p
	\right)$ is zero. The coefficients of $.\left( b_ka_k \right)^p$ in the
	expansion of $d_2\left( .a_{k+1}b_k^p \right)$ and in the expansion of
	$d_2\left( .b_{k+1}a_k^p \right)$ equal $-1$. 
\end{proposition}
\begin{proof}
	We have
	\begin{align*}
		d_2\left( .a_{k+1}b_k^p \right)
		= a_{k+1}.b_k^p -i \left( a_{k+1}d_1\left( .b_k^p \right)
		\right).
	\end{align*} 
	It is easy to check that $d_1\left( .b_k^p \right) = b_k^{p-1}.b_k$. To
	compute $i\left( a_{k+1}b_k^{p-1}.b_k \right)$, we have to find the
	leading term in the normal form of $a_{k+1}b_k^{p-1}$.
	
	Note that the elements $a_{k+1}$ and $b_k$ lie in the subalgebra
	$U_3^{k,k+1}\left( \K \right)$ of $\U_3^+\left( \K \right)$, and the
	homogeneous component of degree $p^{k+1}\alpha + (p-1)p^k\beta$ of
	$\U^{k,k+1}_3\left( \K \right)$ has the basis
	$$
	{a_{k+1}b_k^p} \cup \left\{\, a_k^{p-j}\left( b_ka_k
	\right)^jb_{k}^{p-j-1} \,\middle|\, 0\le j\le p-1
	\right\}. 
	$$

	We get from Proposition~\ref{gbl} and the above remark 
	\begin{align*}
			a_{k+1} b_k^{p-1} & =  b_ka_{k+1}b_k^{p-2} +
				a_k^{p-1}b_ka_kb_k^{p-1}\\& = \dots =
				b_k^{p-1}a_{k+1} +
					\sum_{j=1}^p \lambda_j a_k^{p-j} \left(
				b_ka_k
					\right)^jb_k^{p-j}
				\end{align*}
		for some coefficients $\lambda_j\in \K$. We see that  $b_k^{p-1}a_{k+1}$
	is the maximal term in the expansion of $a_{k+1}b_k^{p-1}$.
	Since $j\left(b_k^{p-1}a_{k+1}.b_k \right) = b_k^{p-1}.a_{k+1}b_k$ we
	get
	\begin{align*}
		d_2\left( .a_{k+1}b_k^{p} \right) &= a_{k+1}.b_k^p -
		b_k^{p-1}.a_{k+1}b_k \\&\phantom{=} +i\left( b_k^{p-1}a_{k+1}.b_k -
		b_k^{p-1}d_1\left( .a_{k+1}b_k \right) + \sum_{j=1}^p \lambda_j
		a_k^{p-j} \left( b_ka_k \right)^j b_k^{p-1-j}.b_k\right).
	\end{align*}
	We have
	\begin{align*}
		d_1\left( .a_{k+1}b_k \right) = a_{k+1}.b_k - b_k.a_{k+1} +
		a_k^{p-1}b_k.a_k.
	\end{align*}
	Therefore
	\begin{align*}
		b_k^{p-1}a_{k+1}.b_k -
		b_k^{p-1}d_1\left( .a_{k+1}b_k \right) &= b_k^p.a_{k+1} -
		b_k^{p-1}a_k^{p-1}b_k.a_k = -b_k^{p-1}a_k^{p-1}b_k.a_k.
	\end{align*}
Since $b_k$ and $a_k$ are the elements of $U_3^{k,k}\left( \K \right)$, the
element $b_k^{p-1}a_k^{p-1}$ is a linear combination of the elements in
$\left\{\, a_k^{p-1-j}\left( b_ka_k \right)^j b_k^{p-1-j} \,\middle|\, 0\le j\le
p-1\right\}$. Note that $\left( b_ka_k \right)^{p-1}$ is maximal among them.
\begin{lemma}
	\label{lemma1}
	The coefficient of $\left( b_ka_k \right)^{p-1}$ in the normal expansion
	of $b_k^{p-1}a_k^{p-1}$ is $1$.
\end{lemma}
\begin{proof}
	Since $\U_3^{k,k}\left( \K \right)$ is isomorphic to $\U_3^{0,0}\left(
	\K \right)$ via $F_k$, and $F_k\left( a_0 \right) = a_k$, $F_k\left(
	b_0 \right) =b_k$, it is enough to prove that the coefficient of
	$\left( ba \right)^{p-1}$ in the normal expansion of $b^{p-1}a^{p-1}$is
	$-1$, where $a:=a_0$ and $b:= b_0$. 

	We have by \eqref{betaalpha}

	\begin{align}
		\label{betaalpha2}
		e_\beta^{\left( p-1 \right)}e_{\alpha}^{\left( p-1 \right)} =
		\sum_{j=0}^p \left( -1 \right)^j e_{\alpha}^{\left( p-1-j
		\right)} e_{\alpha+\beta}^{\left( j \right)}
		e_{\beta}^{\left(p-1-j\right)}.
	\end{align}

	For $0\le j\le p-1$, $e_{\alpha}^{\left( p-1-j
	\right)}e_{\alpha+\beta}^{\left( j \right)} e_{\beta}^{\left( p-1-j
	\right)}$ is a non-zero multiple of $a^{p-1-j}\left( ab-ba
	\right)^jb^{p-1-j}$. Consider 
$a^{p-1-j}\left( ab-ba
	\right)^jb^{p-1-j}$
	as an element of the free algebra $\mathcal{F}$ generated by $a$ and $b$. Then every term in
$a^{p-1-j}\left( ab-ba
	\right)^jb^{p-1-j}$
	is less then $\left( ba \right)^{p-1}$. Since during the Gr\"obner
	reduction process the tems can only decrease, we see that the
	coefficient of $\left( ba \right)^{p-1}$ in the normal expansion of $a^{p-1-j}\left( ab-ba
	\right)^jb^{p-1-j}$ is zero. 

	Recall that $(p-1)!=-1(\mathrm{mod}\  p)$. 
	Thus $e_{\alpha+\beta}^{\left( p-1 \right)} =
	-e_{\alpha+\beta}^{p-1}= - \left( ab-ba \right)^{p-1}$. Again, if we
	consider $\left( ab-ba \right)^{p-1}$ as an element of
	$\mathcal{F}$, the all the terms in $\left( ab-ba \right)^{p-1}$ are
	less then $\left( ba \right)^{p-1}$ except $\left( -1
	\right)^{p-1}\left( ba \right)^{p-1}$. Thus in the normal expansion of
	$\left( ab-ba \right)^{p-1}$
	the term $\left( ba \right)^{p-1}$ enters with the coefficient
	$\left( -1 \right)^{p-1}$. Since the coefficient of
	$e_{\alpha+\beta}^{\left( p-1 \right)}$ in \eqref{betaalpha2} is
	$\left( -1 \right)^p$,  we see that
	the coefficient of $\left( ba \right)^{p-1}$ in the normal expansion of
	$$b^{p-1}a^{p-1} = \left( -e_\beta^{\left( p-1 \right)} \right)\left(
	-e_\alpha \right)^{\left( p-1 \right)} = e_\beta^{\left( p-1
	\right)}e_{\alpha}^{\left( p-1 \right)}$$
	is $-\left( -1 \right)^p\left( -1 \right)^{p-1} = 1$.\end{proof}

Since $\left( b_ka_k \right)^{p-1}b_k.a_k$ is greater then $a_k^{p-j}\left(
b_ka_k
\right)^jb_k^{p-1-j}.b_k$ for any $1\le j\le p$, we see that $\left( b_ka_k
\right)^{p-1}a_k.b_k$
is the maximal term in the normal expansion of 
$$
 b_k^{p-1}a_{k+1}.b_k -
		b_k^{p-1}d_1\left( .a_{k+1}b_k \right) + \sum_{j=1}^p \lambda_j
		a_k^{p-j} \left( b_ka_k \right)^j b_k^{p-1-j}.b_k.
$$
Since $j\left( \left( b_ka_k \right)^{p-1}b_k.a_k \right) = .\left(b_ka_k
\right)^p$, we get
\begin{align*}
	d_2\left( .a_{k+1}b_k^p \right) = a_{k+1}.b_k^p - b_k^{p-1}.a_{k+1}b_k -
	.\left( b_ka_k \right)^{p}  - i\left( \mbox{smaller terms} \right).
\end{align*}
	Now we consider $d_2\left( .b_{k+1}a_k^p \right)$. We have
	\begin{align*}
		d_{2}\left( .b_{k+1}a_k^p \right) & = b_{k+1}.a_k^p - i\left(
		b_{k+1}d_1\left( .a_k^p \right) \right). 
	\end{align*}
	It is easy to check that $d_1\left( .a_k^p \right) = a_k^{p-1}.a_k$.
	Thus
	\begin{align*}
		d_2\left( .b_{k+1}a_k^p \right) = b_{k+1}.a_k^p - i\left(
		b_{k+1} a_k^{p-1}.a_k
		\right).
	\end{align*}
	Now $b_{k+1}a_k^{p-1}$ is an element of $\U_3^{k,k+1}\left( \K
	\right)$ of degree $\left( p-1 \right)p^k \alpha + p^{k+1}\beta$, and
	thus is a linear combination of the elements
	$$
V:= 	\left\{\, a_k^{p-1-j}\left( b_ka_k \right)^jb_k^{p-j} \,\middle|\, 0\le
	j\le p-1\right\} \cup {a_k^{p-1}b_{k+1}}. 
	$$
	Note that for any $v\in V$, we have $va_k<a_{k+1}b_{k+1}$. Therefore,
	$.a_{k+1}b_{k+1}$ enters with the coefficient zero in the expansion of
	$d_2\left( .b_{k+1}a_k^p \right)$. 

	Now $\left( b_ka_k \right)^{p-1}b_k$ is the maximal element of
	$V$.
	\begin{lemma}
		The coefficient of $\left( b_ka_k \right)^{p-1} b_k$ in the
		normal expansion of $b_{k+1}a_k^{p-1}$ is $1$. 
	\end{lemma}
\begin{proof}
		Since $U_3^{k,k+1}\left( \K \right)$ is isomorphic to $U_3^{0,1}\left( \K \right)$ via $F_k$, it is enough to check that the coeffecient of
	$\left( b_0a_0 \right)^{p-1}b_0$ in the normal expansion of
	$b_1a_0^{p-1}$ is $1$. 

We have by \eqref{betaalpha}
\begin{align*}
	e_{\beta}^{\left( p \right)} e_{\alpha}^{\left( p-1 \right)} &= 
	\sum_{j=1}^{p-1}\left( -1 \right)^j e_{\alpha}^{\left( p-j-1 \right)} \left(
	e_{\alpha+\beta}^{\left( j \right)}e_{\beta}^{\left( p-j \right)}
	\right). 
\end{align*}
For $1\le j\le p-2$, $e_{\alpha}^{\left( p-1-j
\right)}e_{\alpha+\beta}^{\left( j \right)}e_{\beta}^{\left( p-j \right)}$ is a
non-zero multiple 
of $$f := a_0^{p-j-1}\left( a_0b_0-b_0a_0 \right)^jb_0^{p-j}.$$
Now the coefficient of $\left( b_0a_0 \right)^{p-1}b_0$ in 
$f$ is zero since all terms of $f$ considered as an element of the free algebra
generated by $a_0$ and $b_0$ are less then $\left( b_0a_0 \right)^{p-1}b_0$. 
For $j=0$, we have
$e_{\alpha}^{\left( p-1 \right)}e_\beta^{\left( p \right)} = - a_0^{p-1}b_1$ and
$a_0^{p-1}b_1< \left( b_0a_0 \right)^{p-1}b_0$. 
For $j=p-1$, we get
\begin{align*}
	e_{\alpha+\beta}^{\left( p-1 \right)}e_\beta = -\left( a_0b_0-b_0a_0
	\right)^{p-1}b_0.
\end{align*}
By the considerations as in Lemma~\ref{lemma1}, we see that the coefficient of
$\left( b_0a_0 \right)^{p-1}b_0$ in 
$\left( a_0b_0-b_0a_0 \right)^{p-1}b_0$ is $\left( -1 \right)^{p-1}$. Therefore
the coefficient of $\left( b_0a_0 \right)^{p-1}b_0$ in 
$$
b_1a_0^{p-1} =  - e_{\beta}^{\left( p \right)}e_{\alpha}^{\left( p-1
\right)}
$$
is $-\left( -1 \right)^{p-1}\left( -\left( -1 \right)^{p-1} \right) = 1$. 
\end{proof}
Since $j\left( \left( b_ka_k \right)^{p-1}a_k.b_k \right) = .\left( b_ka_k
\right)^p$, 
we get
$$
d_2\left( .b_{k+1}a_k \right) = b_{k+1}.a_k - .\left( b_ka_k \right)^p + i\left(
\mbox{smaller terms} \right). 
$$
\end{proof}
\begin{corollary}
	The elements $d_2\left( .b_{k+1}a_k^p \right)+ .\left( b_ka_k \right)^p$
	and $d_2\left( .a_{k+1}b_k^p \right) + .\left( b_ka_k \right)^p$ belong
	to $\Rad\left( P_1 \right)$. 
\end{corollary} 
Now we modify the Anick resolution in order to make it minimal up to the third
step. We define $T'_1= \left\{\, w\in T_1 \,\middle|\, w\not=\left( b_ka_k
\right)^p \right\}$ and $T'_2=\left\{\, w\in T_2 \,\middle|\, w\not=a_{k+1}b_k^p
\right\}$. Denote by $P'_1$ and $P'_2$ the $A$-linear spans of $T'_1$ and
$T'_2$, respectively. 
We define $d'_1\colon P'_1\to P_0$ to be the restriction of $d_1$ to $P'_1$. To
define $d'_2\colon P'_2\to P'_1$ we proceed as follows. For every $w\in T'_2$,
the expression $d_2\left( .w \right)$ can be written as $A$-linear combination 
$$
\sum_{u\in T'_1}r_u.u + \sum_{k}r_k.\left( b_ka_k \right)^p.
$$
Define $d'_2\left( w \right)$ by the expression
$$
\sum_{u\in T'_1}r_u.u + \sum_{k} r_k\left( d_2\left( .a_{k+1}b_k^p
\right)+.\left( b_ka_k \right)^p \right). 
$$
Then it is obvious that the image of $d'_2$ is a subset of $\Rad\left( P'_1
\right)$.

\begin{proposition}
	The complex
$$
P'_2 \stackrel{d'_2}{\longrightarrow} P'_1 \stackrel{d'_1}{\longrightarrow}
	P_0 \stackrel{d_0}{\longrightarrow} P_{-1}
	\stackrel{\varepsilon}{\longrightarrow}\K \to 0
$$
is exact at the term $P'_1$. 
\end{proposition}
\begin{proof}
Denote by $P''_2$ and $P''_1$ the submodules of $P_2$ and $P_1$ generated by
the elements of the form $.a_{k+1}b_k^p$ and $d\left( .a_{k+1}b_k^p \right)$.
It is obvious that $P''_2$
is a free submodule. Since $\left( .b_ka_k \right)^p$ has coefficient $-1$ in
the expansion of $d_2\left( .a_{k+1}b_k^p \right)$, it is easy to show by
induction on $k$, that $P''_1$ is a free submodule of $P_1$. Moreover, it is
obvious that the restriction of $d_2$ on $P''_2$ induces an isomorphism between
$P''_2$ and $P''_1$. Define the $A$-homomorphism $\phi_2\colon P_2\to P'_2$ by $\phi_2\left(.w 
\right) = .w$ if $w\in T'_2$ and  $\phi_2\left( .w \right) =0$ if $w\not \in
T'_2$. Define the $A$-homomorphism $\phi_1\colon P_1\to P'_1$ by $\phi_2\left(
.w \right) = .w$ if $w\in T'_1$ and $\phi_1\left( .\left( b_ka_k \right)^p
\right) = .\left( b_ka_k \right)^p+ d_2\left( .a_{k+1}b_k^p \right)$. Then the
diagram
$$
\xymatrix{
P''_2 \ar[r]^{\cong} \ar[d] & P''_1\ar[r] \ar[d] & 0\ar[d]\\
P_2 \ar[r]^{d_2} \ar[d]^{\phi_2} & P_1 \ar[r]^{d_1} \ar[d]^{\phi_1} &
P_0\ar[d]^{\cong}
\\
P'_2 \ar[r]^{d'_2} & P'_1 \ar[r]^{d'_1} & P_0
}
$$
is commutative, and all its columns are exact.
Therefore the corresponding vertical spectral sequence collapse at stage
$1$. The stage $1$ of the horizontal spectral sequence has the form
$$
\xymatrix{
0 & 0 & 0 \\
\mathrm{Ker}\left( d_2 \right) & 0 & \mathrm{Coker}\left( d_1 \right)\\
\mathrm{Ker}\left( d'_2 \right) & H'_1 & \mathrm{Coker}\left( d'_1 \right)
}
$$
It is easy to see that there is no high differentials that starts or terminate
at $H'_1$. Thus $H'_1$ is zero and $P'_\bullet$ is exact at the term $P'_1$. 
\end{proof}

\bibliography{anick}
\bibliographystyle{amsplain}

\end{document}